\documentclass[11pt]{article}
\usepackage{amsmath, amsfonts, amsthm, amssymb, multicol}
\usepackage{graphicx}
\usepackage{float}
\usepackage{verbatim}
\usepackage{xcolor}
\usepackage{tikz}
\usetikzlibrary{arrows,decorations.markings}

\usepackage[square,sort,comma,numbers]{natbib}
\usepackage{fancyhdr}
 
\numberwithin{equation}{section}

\newtheorem{thm}{Theorem}[section]
\newtheorem{lma}[thm]{Lemma}
\newtheorem{cor}[thm]{Corollary}

\newtheorem{prop}[thm]{Proposition}

\renewcommand{\ge}{\geqslant}
\renewcommand{\le}{\leqslant}
\renewcommand{\geq}{\geqslant}
\renewcommand{\leq}{\leqslant}

\newcommand{\ad}{\dim_\textup{A}}
\newcommand{\hd}{\dim_\textup{H}}

\newcommand{\eps}{\varepsilon}

\allowdisplaybreaks

\title{Attainable values for the  Assouad dimension of projections}

\author{Jonathan M. Fraser \& Antti K\"aenm\"aki}

\begin{document}

\date{}

\maketitle

\begin{abstract}
We prove that for an arbitrary upper semi-continuous function $\phi\colon G(1,2) \to [0,1]$ there exists a compact set $F$ in the plane  such that $\dim_\textup{A} \pi F = \phi(\pi)$ for all $\pi \in G(1,2)$, where $\pi F$ is the orthogonal projection of $F$ onto the line $\pi$.  In particular, this shows that the Assouad dimension of orthogonal projections can take on any finite or countable number of distinct values on a set of projections with positive measure.  It was previously known that two distinct values could be achieved with positive measure.  Recall that for other standard notions of dimension, such as the Hausdorff, packing, upper or lower box dimension, a single value occurs almost surely.
\end{abstract}

\section{Dimensions of Projections}

Starting with the seminal work of Marstrand \cite{marstrand} from the 1950s, the behaviour of dimension under orthogonal projection has received a lot of attention in the fractal geometry and geometric measure theory literature.  The general principle has been that the dimension of the orthogonal projection of a Borel set in $\mathbb{R}^d$ onto a $k$-dimensional subspace almost surely does not depend on the specific choice of subspace.  This has been verified for many different notions of dimension and in more general settings, see for example the surveys \cite{proj1,proj2,proj3} and the recent advance \cite{kaenmakiorponenvenieri}.  The Marstrand-Mattila projection theorem for Hausdorff dimension states that, given a Borel set $F \subseteq \mathbb{R}^d$, we have
\[
\hd \pi F \ = \ \min\{k, \hd F\}
\]
for almost every $\pi \in G(k,d)$, where $\hd$ denotes the Hausdorff dimension, and $G(k,d)$ is the Grassmanian manifold consisting of $k$-dimensional subspaces of $\mathbb{R}^d$.  For notational convenience we identify a subspace $\pi \in G(k,d)$ with the orthogonal projection of $\mathbb{R}^d$ onto that subspace.  It was recently proved by Fraser and Orponen \cite{fraserorponen} in the planar case and   Fraser   \cite{fraser} in the general case that for any set $F \subseteq \mathbb{R}^d$, we have
\begin{equation} \label{assprojthm}
\ad \pi F \ \geq \ \min\{k, \ad F\}
\end{equation}
for almost every $\pi \in G(k,d)$, where $\ad$ denotes the Assouad dimension, see below for the definition. Very recently, Orponen \cite{orponen} has improved the planar case of \eqref{assprojthm} by showing that the inequality holds for all $\pi \in G(1,2)$ outside of a set of Hausdorff dimension zero. The surprising additional result from \cite{fraserorponen} is  that one cannot improve the almost sure inequality above  to an almost sure equality.  In particular, it was shown in  \cite{fraserorponen} that $\ad \pi F$ can take on two distinct values with positive measure.  The purpose of this article is to push this observation further: \emph{how wild can $\ad \pi F$ be as a function of $\pi$?}

To finish the introduction, we recall the definition of the Assouad dimension.  For any non-empty bounded set $E \subset \mathbb{R}^d$ and $r>0$, let $N_r (E)$ be the smallest number of open sets with diameter less than or equal to $r$ required to cover $E$.  The \emph{Assouad dimension} of a non-empty set $F\subseteq \mathbb{R}^d$ is defined by
\begin{eqnarray*}
\dim_\text{A} F & = &  \inf \Bigg\{ \  s \geq 0 \  : \ \text{      $ (\exists \, C>0)$ $(\forall \, R>0)$ $(\forall \, r \in (0,R) )$  } \\ 
&\,& \hspace{35mm} \text{ $\sup_{x \in F}N_r\big( B(x,R) \cap F \big) \ \leq \ C \bigg(\frac{R}{r}\bigg)^s$ } \Bigg\}
\end{eqnarray*}
where $B(x,R)$ denotes the closed ball centred at $x$ with radius $R$. It is well-known that the Assouad dimension is always an upper bound for the Hausdorff dimension and upper box dimension; for example, see \cite{luukkainen}.

\section{Results}

Our main result allows one to construct compact sets in the plane whose projections behave in a predefined manner with respect to the Assouad dimension.

\begin{thm} \label{main1}
Given an upper semi-continuous function $\phi\colon G(1,2) \to [0,1]$, there exists a compact set $F$ in the plane  with $\ad F = 0$ such that $\ad \pi F = \phi(\pi)$ for all $\pi \in G(1,2) $. 
\end{thm}


We remark that the assumption on upper semi-continuity is only needed to prove the upper bound $\ad \pi F \le \phi(\pi)$ in the claim.  See Section \ref{discussion} for a discussion of the sharpness of this result and possible future directions.

In the context of Marstrand's projection theorem, that is, the approach of describing the dimensions of generic projections, we obtain the following corollaries, which are in stark contrast to what is possible for other standard notions of dimension, such as the Hausdorff, packing or box dimensions.

\begin{cor} \label{cor1}
Let $E \subseteq [0,1]$ be any finite or countable set.  There exists a compact set $F$ in the plane  such that for all $s \in E$, we have $\dim_\textup{A} \pi F = s$ for a set of $\pi \in G(1,2) $ of  positive measure.
\end{cor}

\begin{proof}
This follows immediately from Theorem \ref{main1} by letting $\{I_x\}_{x \in E}$ be a family of pairwise disjoint compact subsets of $G(1,2)$ each with non-empty interior and then choosing $\phi$ to satisfy $\phi(\pi) = x$ for all $\pi \in I_x$ and $\phi(\pi) = 0$ otherwise. 
\end{proof}

This answers \cite[Question 2.2]{fraserorponen} which asked how many distinct values the Assouad dimension
of $\pi F$ can assume for a set of $\pi$ with positive measure.  This corollary also answers  \cite[Question 2.7]{fraserorponen}, which asked if there exists a  compact planar set for which the Assouad dimension of the projection takes different values on two sets with non-empty interior.  Finally, we can answer \cite[Question 2.6]{fraserorponen} which asked if, given $0\leq s < \log_5 3$, there exists a compact planar set with Hausdorff dimension $s$, for which the Assouad dimension of the projections is not almost surely constant.  This was motivated by the fact that the examples constructed in \cite{fraserorponen} necessarily had Hausdorff dimension at least $\log_5 3$.  The examples constructed in Theorem \ref{main1} provide such a set when $s=0$, which was the most difficult case, and to adapt this for $s>0$ one can simply add to $F$ a self-similar set of Hausdorff and Assouad dimension $s$ which is contained in a line.  The projections of this set onto  all but one subspace will also have dimension $s$ and therefore will only influence the Assouad dimensions of projections where $\ad \pi F < s$.

In addition to finding many values which appear as the Assouad dimension of projections with positive probability, we can also provide the following complementary result which \emph{avoids} all values almost surely.

\begin{cor}
There exists a compact set $F$ in the plane  such that, for all $s \in [0,1]$, there is at most one $\pi$ such that $\dim_\textup{A} \pi F=s$.
\end{cor}
\begin{proof}
This follows immediately from Theorem \ref{main1} by letting $\phi$ be given by $\phi(x)=x/\pi$ where we identify $G(1,2)$ with the interval $(0,\pi]$ in the natural way such that the projection onto the horizontal axis is identified with $\pi \in (0,\pi]$.
\end{proof}

\section{Construction of $F$ and proof of Theorem \ref{main1}}

\subsection{A quantitative  separability property for functions}

We require the following technical lemma which we state in more generality than we need since it may be of interest in its own right and we could not find it in the literature.  The result should be thought of as a quantitative  separability property for functions.  The upper box dimension, $\overline{\dim}_\text{B} X$, of a metric space $X$ may be defined as the infimum of $s$ such that for all $r>0$, there exists a cover of $X$ by at most $r^{-s}$  balls of radius $r$ centred in $X$.  The upper box dimension may be infinite, but it is finite for any bounded subset of Euclidean space, for example.

\begin{lma} \label{countable-set}
Let $X$ and $Y$ be  metric spaces with finite upper box dimensions. For any function $\phi \colon X \to Y$ there exists a countable dense subset $Q=\{q_1, q_2, \dots\}$ of $X$ such that, given $x \in X$, there exists a sequence $(q_{n_k})_k$ of points in $Q$ for which  $\phi(q_{n_k}) \to \phi(x)$ as $k \to \infty$,  and $q_{n_k} \to x$ with
\[
d_X(q_{n_k},x) \ \leq \ n_k^{-1/s}
\]
where $d_X$ is the metric on $X$ and $s=\overline{\dim}_\textup{B} X+\overline{\dim}_\textup{B} Y+2$.
\end{lma}

\begin{proof}
Consider the set $\text{Graph}(\phi) := \{(x,\phi(x)) : x \in X\} \subseteq X \times Y$, where the product is equipped with the sup metric $d_{\infty}$ and note that an easy argument shows $\overline{\dim}_\textup{B} \text{Graph}(\phi) \leq \overline{\dim}_\textup{B} X+\overline{\dim}_\textup{B} Y<s-1$. For integer $k \geq 1$ let $Q'_k \subseteq \text{Graph}(\phi) $ be the set of centres of balls of radius $1/k$ in a cover of  $\text{Graph}(\phi)$ with  $\# Q_k' \leq k^{s-1}$.  The set $Q' := \bigcup_{k \geq 1} Q'_k$  is a countable dense subset of   $\text{Graph}(\phi)$.  Moreover, if $p_X \colon X \times Y \to X$ is the canonical projection, then let $Q=p_X(Q')=\{q_1, q_2, \dots\}$ where the ordering is such that points in $p_X(Q'_k)$ are labelled before points in $p_X(Q'_{k+1})$ for all $k \geq 1$ (up to multiplicity).  The set $Q$ (with the given labelling) satisfies the requirements of the lemma.  Clearly $Q$ is dense in $X$ and, moreover, given $x \in X$ and $k \geq 1$, choose  $q_{n_k} =  p_X(q'_k)$ for some $q_k' \in Q_k'$ such that $d_{\infty}(q'_{n_k},(x,\phi(x))) \leq1/k$. It follows that
\[
 d_Y(\phi(q_{n_k}), \phi(x)) \ \leq \ 1/k \ \to \ 0 
\]
as $k \to \infty$, where $d_Y$ is the metric on $Y$.  Moreover, by the above labelling procedure,
\[
n_k \ \leq \ \sum_{m=1}^k \# Q_m' \ \leq \ \sum_{m=1}^k  m^{s-1} \ \leq \ k^{s}
\]
and therefore
\[
d_X(q_{n_k},x) \ \leq \ 1/k \leq n_k^{-1/s}
\]
as required.
\end{proof}

\subsection{Construction of $F$}

From now on, we fix  an upper semi-continuous function $\phi \colon G(1,2) \to [0,1]$.  Let $\Pi = \{\pi_1, \pi_2, \dots\}$ be a countable dense subset of $G(1,2)$ satisfying the requirements of Lemma \ref{countable-set}. For each $\pi_k\in \Pi$ and integer $n \geq 1$, we construct a finite set $F_{\pi_k,n}$ associated to $\pi_k$ as follows.  Let $s=\phi(\pi_k)$ and choose $c \in (0,1/2]$  such that $\log2/\log(1/c) = s$ provided $s>0$ and $c=0$ otherwise.  Let
$S_0, S_1 \colon [0,1]\to [0,1]$ be defined by $S_0(x) = cx$ and $S_1(x) = cx+1-c$ and let
\[
E_n(c) \ = \ \{ S_{i_1} \circ \dots \circ S_{i_n} (0) \ : \ i_j \in \{0,1\} \} \ \subset \ [0,1]
\]
which is a finite approximation of the self-similar set generated by $\{S_0,S_1\}$.  We assume a basic familiarity with self-similar sets; see for example \cite[Chapter 9]{falconer}.  Note that $E_n(c)$ consists of $2^n$ points, provided $c>0$, and 2 points if $c=0$.   If $c>0$, let
\[
Z_n \ = \ \{ 2^{-i} \ : \ i=1, 2, \dots, 2^{n+1}-2\}
\]
and otherwise let $Z_n=\{0,1\}$.   If $c>0$, let
\[
Y_n (c) \ = \ \bigcup_{m=1}^n \left( 9^{-2^m}+16^{-2^m}E_m(c)\right)
\]
and otherwise let $Y_n (c) =\{0,1/2\}$.  Note that if we write     $Y_n(c) = \{y_1, y_2, \dots\}$ where the points are labelled in decreasing order then we have
\begin{equation} \label{yydecay}
y_i \ \leq \ 3^{-i}+4^{-i}.
\end{equation}
Observe that $Z_n$ and $Y_n(c)$ are finite sets with the same cardinality.  Let $\iota \colon Y_n(c) \to Z_n$ be the unique increasing  bijection between $Y_n(c)$ and $Z_n$ and let
\[
F'_n \ = \ \{(y, \iota(y) ) : y \in Y_n(c)\}.
\]
Note that $F'_n \subseteq Y_n(c) \times Z_n$,  $\pi^1 F'_n = Y_n(c)$ and $\pi^2 F'_n = Z_n$, where $\pi^i$ denotes projection onto the $i$th coordinate.   The idea here is that $\pi^1 F'_n = Y_n(c)$ is  approximating a set which itself is approximating a  self-similar set of dimension $s$ and  $\pi^2 F'_n = Z_n$ is approximating the set 
\[
Z \ = \ \{0\} \cup \bigcup_n Z_n \ = \ \{0\} \cup \{ 2^{-i} : i=1, 2, \dots\}
\]
 which has Assouad dimension $0$. Specifically, the sets $Y_n(c)$  are a nested increasing sequence of sets approximating the set
\[
Y(c) \ = \ \bigcup_{n \geq 1} Y_n(c)
\]
which has Assouad dimension $s$.  We now want to force the projections of $F'_n$ to also appear $0$-dimensional for projections  distinct from $\pi^1$ and $\pi^2$, which is achieved by `stretching', and to re-align $F_n'$ such that the subspace corresponding to the  large projection (in terms of dimension)  is $\pi_k$ instead of $\pi^1$, which is achieved by rotating. Finally, we shall place the stretched and rotated versions of $F'_{n}$ on the graph of $x \mapsto x^2$ in such a way that their projections (onto any subspace) do not overlap with each other too much. To this end, let  $h$ and $v$ be decreasing functions from $\mathbb{N}$ to $(0,1)$ chosen such that  
\[
h(i) \ = \ \frac{v(i)}{\log i}  \qquad \text{and} \qquad \frac{v(i)}{10^{-i}} \ \to \ 0.
\]
Let $g \colon \Pi \times \mathbb{N} \to \mathbb{N}$ be a bijection satisfying $  n \leq g(\pi_k, n) \leq \max\{n,k\}^2$ for all $(\pi,n) \in\Pi \times \mathbb{N}$.  Such a $g$ is easily constructed, recall for example the standard enumeration of the positive rationals.
\begin{figure}[t]
\centering
\begin{tikzpicture}[scale=14]
  \begin{scope}[font=\scriptsize]
    \draw [->] (-0.01,0) -- (0.6,0) node [] {};
    \draw [->] (0,-0.01) -- (0,0.3) node [] {};

    \draw (0.0078125,0.005) -- (0.0078125,-0.005);
    \draw (0.015625,0.005) -- (0.015625,-0.005)   node [below] {$\cdots$};
    \draw (0.03125,0.005) -- (0.03125,-0.005);
    \draw (0.0625,0.005) -- (0.0625,-0.005)   node [below] {$2^{-4}$};
    \draw (0.125,0.005) -- (0.125,-0.005)   node [below] {$2^{-3}$};
    \draw (0.25,0.005) -- (0.25,-0.005)   node [below] {$2^{-2}$};
    \draw (0.5,0.005) -- (0.5,-0.005)   node [below] {$2^{-1}$};

    \draw (0.005,0.25) -- (-0.005,0.25)   node [left] {$4^{-1}$};
    \draw (0.005,0.0625) -- (-0.005,0.0625)   node [left] {$4^{-2}$};
    \draw (0.005,0.015625) -- (-0.005,0.015625)   node [left] {$4^{-3}$};
    \draw (0.005,0.00390625) -- (-0.005,0.00390625);
  \end{scope}

  \draw [dotted,domain=0:0.5,smooth,variable=\x] plot ({\x},{(\x)*(\x)});

  \draw [xshift=0.5 cm,yshift=0.25 cm,rotate=30,scale=0.06] (0,0) rectangle (1,3);
  \draw [xshift=0.25 cm,yshift=0.0625 cm,rotate=-20,scale=0.03] (0,0) rectangle (1,4);
  \draw [xshift=0.125 cm,yshift=0.015625 cm,rotate=10,scale=0.01] (0,0) rectangle (1,5);
  \draw [xshift=0.0625 cm,yshift=0.00390625 cm,rotate=25,scale=0.004] (0,0) rectangle (1,6);

  \end{tikzpicture}
  \caption{Illustration for the construction of $F$. Here each rectangle represents one of the sets  $F_{\pi,n}$. Note that, for illustrative purposes, the rectangles are not to scale.}
  \label{fig:illustration}
\end{figure}
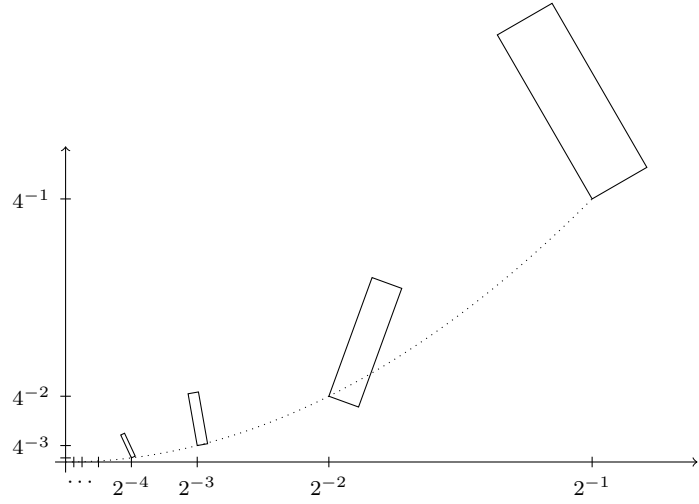
Let
\begin{equation} \label{eq:F-def}
F_{\pi,n} \ = \ R_\pi  \circ A_{g(\pi,n)} ( F_n')+t_{g(\pi,n)}
\end{equation}
where $A_{i}$ is the linear map which scales by $h(i)$ in the horizontal direction and $v(i)$ in the vertical direction,  $R_\pi$ is the rotation about the origin so that the base of $R_\pi([0,1]^2)$ is parallel with $\pi$, and $t_{i} = (2^{-i}, 4^{-i})$ is a translation.  Finally, let
\begin{equation*}
  F \ = \ \{(0,0)\} \cup \bigcup_{(\pi,n) \in \Pi \times \mathbb{N}}  F_{\pi,n}
\end{equation*}
and notice that, by construction, the set $F$ is  compact.

\subsection{Some preliminaries: covering estimates and weak tangents }

Before showing that the set $F$ satisfies the properties claimed in Theorem \ref{main1}, we provide the reader with some preliminaries.

\begin{lma} \label{coveringFk}
For all $\eps>0$ there exists a constant $C_\eps \ge 1$ such that for every $\pi\in \Pi$, $n \geq 1$, $x \in \mathbb{R}^2$, and $0<r<R$ we have $N_r(B(x,R) \cap F_{\pi,n}) \leq C_\eps (R/r)^\eps$.
\end{lma}

\begin{proof}
Fix $\eps>0$, and let $C'_\eps \ge 1$ be such that for all $z \in \mathbb{R}$ and $0<r<R$ we have
\begin{equation}\label{coveringZ}
N_r(B(z,R) \cap Z) \ \leq \ C_\eps'(R/r)^\eps.
\end{equation}
This can be achieved since $\ad Z = 0$.  Fix $\pi \in \Pi$ and $n \geq 1$, and let $\pi^\perp$ denote the orthogonal complement of $\pi$.   For distinct $x,y \in F_{\pi,n}$, note that
\[
|\pi^\perp(x)-\pi^\perp(y)| \ = \ (2^{-i}-2^{-j})v(g(\pi,n)) \ \geq \ \frac{v(g(\pi,n))}{2^{i+1}}
\]
for some integers $j>i\geq 1$. Therefore, using \eqref{yydecay}, we also have
\[
|\pi(x)-\pi(y)| \ \leq \ (3^{-i}+4^{-i}) h(g(\pi,n)) \ = \ (3^{-i}+4^{-i}) \frac{v(g(\pi,n))}{\log(g(\pi,n))} \ \leq \ 2 \frac{v(g(\pi,n))}{3^i}.
\]
This shows that $\pi^\perp$ is a bi-Lipschitz map on $F_{\pi,n}$ with lower Lipschitz constant bounded away from $0$  independently of $\pi$ and $n$  and upper Lipschitz constant trivially bounded above by 1.  Therefore there is a uniform constant $C_0 \geq 1$ such that for all $r>0$ and any subset $F' \subseteq F_{\pi,n}$ we have $N_r(F') \leq C_0 N_r(\pi^\perp F')$. Note that  $\pi^\perp F_{\pi,n}$ is a subset of $Z$ scaled down by a factor of $v(g(\pi,n))$.  It follows that for arbitrary $x \in \mathbb{R}^2$ and $0<r<R$ we have
\begin{eqnarray*}
N_r(B(x,R) \cap F_{\pi,n}) &\leq& C_0 N_r(B(\pi^\perp(x),R) \cap \pi^\perp F_{\pi,n}) \\ 
&\leq& C_0 \sup_{z \in \mathbb{R}} N_{ r/v(g(\pi,n))} (B(z, R/v(g(\pi,n))) \cap Z ) \\
&\leq&  C_0 C_\eps' (R/r)^{\eps} \qquad \text{by \eqref{coveringZ}}
\end{eqnarray*}
which proves the claim.
\end{proof}

Let $\mathcal{K}(\mathbb{R})$ denote the set of all compact subsets of $ \mathbb{R}$, which is a complete metric space when equipped with the \emph{Hausdorff metric} $d_\mathcal{H}$ defined by
\[
d_\mathcal{H} (A,B) \ = \ \inf \{ \delta \ : \ A \subseteq B_\delta \text{ and } B \subseteq A_\delta \}
\]
where, for any $C \in \mathcal{K}(\mathbb{R})$,
\[
C_\delta \ = \ \{ x \in \mathbb{R} \ : \ | x- y | < \delta \text{ for some } y \in C \}
\]
denotes the open $\delta$-neighbourhood of $C$. One of the most effective ways to bound the Assouad dimension of a set from below is to use  \emph{weak tangents}; an approach going back to Mackay and Tyson \cite{mackaytyson} and Keith and Laakso \cite{keithlaakso}. Let $F \subset \mathbb{R}$ and $E \subseteq B(0,1)$ be compact.  If there exists a sequence of similarity maps $T_k$ on $\mathbb{R}$ such that $d_\mathcal{H} (E,T_k(F) \cap B(0,1) ) \to 0$ as $k \to \infty$, then $E$ is called a \emph{weak tangent} to $F$.

\begin{prop}\label{tangents}
If $F \subseteq \mathbb{R}$ is compact and  $E$ is a weak tangent to $F$, then $\dim_\mathrm{A} F \geq \dim_\mathrm{A} E$. 
\end{prop}

The proof of this proposition can be found in \cite[Proposition 6.1.5]{mackaytyson}. We also have the following dual result which shows that the Assouad dimension can always be achieved as the Hausdorff dimension of a weak tangent.

\begin{prop}\label{tangents2}
If $F \subseteq \mathbb{R}$ is compact, then there exists a weak tangent $E$ to $F$ such that $ \dim_\mathrm{H} E = \dim_\mathrm{A} E = \dim_\mathrm{A} F$. 
\end{prop}

This statement follows from \cite[Theorem 5.1]{Fu} and \cite[Proposition 3.13]{KR}, or, alternatively, directly from \cite[Proposition 5.7]{KOR}.

\subsection{Proof of Theorem \ref{main1}}

Let us first show that $\ad F = 0$.
Fix $x \in F$ and $0<r<R \leq 1$, and let $m:=\min\{ g(\pi,n) : B(x,R) \cap F_{\pi,n}  \neq \emptyset\}$.  Note that we may assume $R \leq 1$ since $F$ is bounded.  Suppose first that $B(x,R)$ only intersects one of the sets $F_{\pi,n}$, which is therefore necessarily the one associated to $m$.  It follows from Lemma \ref{coveringFk} that
\[
N_r(B(x,R) \cap F) \ \leq \ N_r(B(x,R) \cap F_{\pi,n}) \ \leq \ C_\eps (R/r)^\eps.
\]
Now suppose that $B(x,R)$  intersects more than one  of the sets $F_{\pi,n}$. In particular, it must intersect $F_{\pi,n}$ for some $g(\pi,n)>m$.  This forces
\[
 2^{-m-3} \ \leq \ R \ \leq \ 2^{-m+2}
\]
and therefore\begin{equation} \label{mRbound}
\frac{-\log R}{\log 2} - 3 \ \leq \ m \ \leq \ \frac{-\log R}{\log 2} + 2.
\end{equation}
We have 
\begin{eqnarray*}
N_r(B(x,R) \cap F)  & = &  N_r \left(B(x,R) \cap \bigcup_{g(\pi,n)=m}^\infty F_{\pi,n}\right) \\ \\ 
& \leq & N_r \left(B(x,R) \cap \bigcup_{g(\pi,n)=\lceil\frac{-\log r}{\log 2}\rceil}^\infty F_{\pi,n}\right) + \sum_{g(\pi,n)=m}^{\lfloor\frac{-\log r}{\log 2}\rfloor} N_r(B(x,R) \cap F_{\pi,n})   \\ \\ 
& \leq & 2+ \sum_{g(\pi,n)=m}^{\lfloor\frac{-\log r}{\log 2}\rfloor} C_\eps \left( \frac{R}{r} \right)^\eps  \qquad \text{by Lemma \ref{coveringFk}} \\ \\ 
& \leq & 2 + \left( \frac{\log R/r}{\log 2} + 4 \right) C_\eps \left( \frac{R}{r} \right)^\eps  \qquad \text{by \eqref{mRbound}} \\ \\
&\leq& (6 C_\eps+\tilde C_\eps )\left( \frac{R}{r} \right)^{2\eps}
\end{eqnarray*}
where $\tilde C_\eps $ is a constant depending only on $\eps$, such that $C_\eps\log x/\log 2 \leq \tilde C_\eps x^\eps$ for all $x \geq 1$.  We have proved that $\ad F \leq 2\eps$ and since $\eps>0$ is arbitrary, we get $\ad F = 0$ as required.


We now prove that $\ad \pi F \ge \phi(\pi)$ for all $\pi \in G(1,2)$. Fix $\pi \in G(1,2)$ and let $s=\phi(\pi)$ which we may assume is strictly positive since otherwise there is nothing to prove.  Let $c \in (0,1/2]$ be such that $s = -\log 2 /\log c > 0$. Let $(\pi_{n_k})_k$ be a sequence of points in $\Pi$ given by Lemma \ref{countable-set} such that  $\phi(\pi_{n_k}) \to \phi(\pi)$ as $k \to \infty$ and $\pi_{n_k} \to \pi$ with
\[
|\pi_{n_k} - \pi | \ \leq \ n_k^{-1/4}.
\]
Here we equip $G(1,2)$ with the arc length metric so that $|\pi_{n_k} - \pi |$ is equal to the smaller of the two angles formed by the subspaces $\pi_{n_k}$ and $\pi  $.  

Let $F_k := F_{\pi_{n_k},k}$ denote the set consisting of $2^{k+1}-2$ points which is associated to $\pi_{n_k}$; see \eqref{eq:F-def} for the precise definition. Further, let $c_k \in (0,1/2]$ be the constant associated with $F_k$ such that $-\log 2/\log c_k = \phi(\pi_{n_k})$. As $\phi(\pi_{n_k}) \to \phi(\pi)$, it follows that $c_k \to c$ as $k \to \infty$.

Let $I = [0,1]$ and, for $\lambda \in (0,1/2]$, let $E(\lambda) \subseteq I$ be the self-similar set associated with the maps $x \mapsto \lambda x$ and $x \mapsto \lambda x+(1-\lambda)$. For each $k \geq 1$ let $T_k \colon \mathbb{R} \to \mathbb{R}$ be the similarity given by
\[
T_k(x) \ = \ |\pi F_k|^{-1}  \left( x -\inf \pi F_k \right)
\]
where $| \cdot |$ denotes diameter of a set. Here we identify a given orthogonal projection of $\mathbb{R}^2$ with  $\mathbb{R}$ in the natural way. We claim that
\begin{equation} \label{eq:conv-to-Ec}
  T_k( \pi F_k ) \ \to \ Y(c)
\end{equation}
 in the Hausdorff metric as $k \to \infty$.   

First note that
\begin{eqnarray*}
1 & \leq & \frac{|\pi  F_k|}{|\pi_{n_k} F_k|} \ \leq \ 1+  \frac{\sin(|\pi_{n_k} - \pi |)v(g(\pi_{n_k},k) )}{h(g(\pi_{n_k},k))} \\
& \leq & 1+ \frac{\log(g(\pi_{n_k},k))}{n_k^{1/4}} \ \leq \ 1+\frac{2 \log n_k}{n_k^{1/4}} \ \to \ 1
\end{eqnarray*}
as $k \to \infty$.  This shows that
\[
d_\mathcal{H}\left( T_k( \pi F_k)  , T_k( \pi_{n_k} F_k )  \right)  \ \to \ 0
\]
as $k \to \infty$.  Moreover, noting that  $T_k( \pi_{n_k} F_k )  \subseteq Y(c)$, 
\[
d_\mathcal{H}\left( T_k( \pi_{n_k} F_k )  , Y(c)  \right)   \ \to \ 0
\]
and  therefore, by applying the triangle inequality, 
\[
d_\mathcal{H}\left(T_k( \pi F_k)  , Y(c) \right) \ \to \ 0
\]
as $k \to \infty$, proving \eqref{eq:conv-to-Ec}.

Since $\mathcal{K} \left(I\right)$ is compact we may extract a  subsequence of the sequence $T_k( \pi F ) \cap I$ which converges to a set $\widehat{\pi F}$ which is necessarily a weak tangent to $\pi F$.  Since $T_k( \pi F_k) \subseteq T_k( \pi F ) \cap I$ for all $k$ by the definition of $F$ and $T_k( \pi F_k )\to Y(c)$ as $k \to \infty$ by \eqref{eq:conv-to-Ec}, it follows that $Y(c) \subseteq \widehat{\pi F}$.  Therefore Proposition \ref{tangents} yields
\[
\ad \pi F \ \geq \ \ad \widehat{\pi F} \ \geq \ \ad Y(c) \ = \ s
\]
as required.  The final claim that $\ad Y(c) \ = \ s$ follows since by construction $E(c)$ is   a weak tangent to $Y(c)$. We could have obtained $E(c)$ directly as a subset of a weak tangent of $\pi F$ but we find the above argument more straightforward.


Finally, we prove that $\ad \pi F \le \phi(\pi)$ for all $\pi \in G(1,2)$.  Fix  $\pi \in G(1,2)$ from now on.  We first argue that the projections $\pi F_{\pi',n}$ ($\pi' \in \Pi$, $n \geq 1$) are exponentially separated, i.e., they satisfy \eqref{separation}.  First suppose that $\pi \neq \pi^2$, where as above $\pi^2$ is projection onto the second coordinate. Since
\[
\frac{|F_{\pi',n}|}{2^{-g(\pi',n)}} \ \leq \ \frac{\sqrt{2} v(g(\pi',n))}{2^{-{g(\pi',n)}}} \ \to \ 0
\]
we have that the sets $\pi(F_{\pi',n})$ are pairwise disjoint for large enough $g(\pi',n)$ (or $n$), and even separated by a distance of at least a constant times $2^{-g(\pi',n)}$. More precisely, there exists a constant  $K \geq 1$  (which may depend on $\pi$)  such that, for all $(\pi', n) \in \Pi \times \mathbb{N}$,
\begin{equation*}
\begin{split}
\inf\{ |x-y| \ : \ x \in \pi F_{\pi'',m}, \; y \in \pi F_{\pi',n}, \text{ and } g(\pi'',m) > \;&g(\pi',n) \geq K\} \\ &\geq \ 2^{-1}|\cos(\theta)| 2^{-g(\pi',n)}
\end{split}
\end{equation*}
 where $\theta$ is the angle  $\pi$ makes with the horizontal axis.  Note that  $\cos(\theta) \neq 0$ since $\pi \neq \pi^2$. The angle $\pi$ makes with the horizontal axis is relevant because the translations $t_i$ lie on the graph of $x \mapsto x^2$, which is tangent to the horizontal axis.  Secondly, suppose that $\pi = \pi^2$. Similar to above, since
\[
\frac{|F_{\pi',n}|}{4^{-g(\pi',n)}} \ \leq \ \frac{\sqrt{2} v(g(\pi',n))}{4^{-{g(\pi',n)}}} \ \to \ 0
\]
we have that the sets $\pi F_{\pi',n}$ are pairwise disjoint for large enough $g(\pi',n)$ (or $n$), and even separated by a distance of at least a constant times $4^{-g(\pi',n)}$.  We are forced to replace $2^{-g(\pi',n)}$ with $4^{-g(\pi',n)}$ here because the $\pi^2$-projection of the set of translations $t_i$ is the sequence $4^{-i}$, rather than (asymptotic to) $\cos(\theta) 2^{-i}$.  Since each of the sets $F_{\pi',n}$ are finite and the Assouad dimension is finitely stable, we may assume without loss of generality that these separation properties holds for all $F_{\pi',n}$.   That is, we may assume that for all $(\pi',n) \neq (\pi'',m)$ with $g(\pi'',m) > g(\pi',n)$ we have
\begin{equation} \label{separation}
\inf\{ |x-y| \ : \ x \in \pi F_{\pi'',m}, \; y \in \pi F_{\pi',n}\}  \ \geq \ C(\pi) 4^{-g(\pi',n)}
\end{equation}
for a constant $C(\pi)$ depending only on $\pi$.

Suppose $E$ is a weak tangent to $\pi F$.  By Proposition \ref{tangents2}, it suffices to show that $\dim_\textup{H} E\leq \phi(\pi)$.  Since $E$ is a weak tangent,  there exists a sequence of similarity maps $T_l$ on $\mathbb{R}$ such that $T_l(\pi F) \cap B(0,1) \to E$ in the Hausdorff metric as $l \to \infty$. We may assume that there are arbitrarily large $k$ such that $T_l(\pi F_{\pi',k}) \cap B(0,1)  \neq \emptyset$  for some $l$ and some $\pi' \in \Pi$.  Otherwise $E$ is a weak tangent of a finite set and would therefore be finite itself. We may also assume that $\|T_l\| \to \infty$ since otherwise $E$ is countable. Here $\|T\|$ denotes the similarity ratio of a similarity map $T$, which is the operator norm of the linear part of $T$, hence the notation.   Therefore, by taking a subsequence of $(T_l)_l$ we may assume that
\begin{equation} \label{growth}
\min\{k \ : \ T_l(\pi F_{\pi',k}) \cap B(0,1)  \neq \emptyset \text{ for some $\pi'$}\} \ \to \ \infty
\end{equation}
as $l \to \infty$.

For a given $l \geq 1$, suppose for some large $k \geq 1$ we have $T_l(\pi F_{\pi',k}) \cap B(0,1)  \neq \emptyset$ for some $\pi'$ and $|T_l(\pi F_{\pi',k})| \geq 2^{-g(\pi',k)}$. It follows that
\[
\|T_l\| \sqrt{2} v(g(\pi',k)) \ \geq \ \|T_l\|  |\pi F_{\pi',k}| \ = \ |T_l(\pi F_{\pi',k})| \ \geq \ 2^{-g(\pi',k)}
\]
and therefore
\[
\|T_l\|^{-1} \ \leq \ 2^{g(\pi',k)} v(g(\pi',k))\sqrt{2} \ < \ 5^{-g(\pi',k)}\sqrt{2}.
\]
It follows from the separation property \eqref{separation} that $T_l(\pi F_{\pi'',i}) \cap B(0,1)  = \emptyset$ for all $(\pi'',i) \neq (\pi',k)$, provided that $k$ is sufficiently large, depending on $\pi$.

On the other hand, suppose that for all large enough $l$ we have $|T_l(\pi F_{\pi',k})| < 2^{-g(\pi',k)}$ for all $k$ and $\pi'$ such that $T_l(\pi F_{\pi',k}) \cap B(0,1)  \neq \emptyset$.  This means that the contribution from individual $\pi(F_{\pi',k})$ vanishes to a single point in the limit and $E$ must be a weak tangent of the projection of the closure of the set of translations, $\{\pi(t_i) \}_i \cup \{0\}$,  and therefore has Hausdorff dimension $0$.  Therefore we may assume (by taking a subsequence if necessary) that for all $l$ there exists a unique pair $(\pi(l), k(l))$ such that $T_l(\pi F_{\pi(l),k(l)}) \cap B(0,1)  \neq \emptyset$ and, moreover, that $|T_l(\pi F_{\pi(l),k(l)})| \geq 2^{-g(\pi(l),k(l))}$. Therefore,
\begin{equation} \label{convtotang}
T_l(\pi F_{\pi(l),k(l)}) \cap B(0,1) \ \to \ E.
\end{equation}
Using  compactness of $G(1,2)$, we may assume that $\pi(l) \to \pi^* \in G(1,2)$ as $l \to \infty$.  If $\pi^* \neq \pi$, then   $E$ is a weak tangent of $Z = \{0\} \cup \{ 2^{-i} : i=1, 2, \dots\}$ and therefore has dimension 0.  The case when $\pi^* = \pi$  is more delicate and is interestingly the only point in the proof where we use upper semi-continuity of $\phi$.

Consider $F_l := F_{\pi(l),k(l)}$ which, for large $l$, is contained in a very long and thin rectangle with long side almost orthogonal to the subspace $\pi$.  Let
\[
Y_l'(c) \ = \ R_{\pi(l)}  \circ A_{g(\pi(l),k(l))} ( Y_{k(l)}(c) \times \{0\})+t_{g(\pi(l),k(l))}
\]
which, by recalling \eqref{eq:F-def}, is simply the projection of $F_l$ onto the base of the rectangle containing it.  Similarly, let 
\[
Z_l' \ = \ R_{\pi(l)} \circ A_{g(\pi(l),k(l))} ( \{0\} \times Z_{k(l)})+t_{g(\pi(l),k(l))}
\]
which is simply the projection of $F_l$ onto the long side of the rectangle containing it.  Order the points in  $F_l$ by writing $F_l =\{x_1, x_2, \dots, x_{2^{k(l)+1}-2}\}$ in the natural `from top to bottom' order. For each $x_n \in F_l$ label the corresponding point in $Y_l'(c)$ by $y_n$ and the corresponding point in $Z_l'$ by $z_n$. 

The ratio
\[
\frac{|\pi F_{\pi(l),k(l)}|}{|\pi(l) F_{\pi(l),k(l)}|} \geq 1
\]
again plays a key role in determining the possible tangents. By taking a subsequence if necessary we may assume that either this ratio is uniformly bounded away from 1 from below or converges to 1.  First assume the former. In this case we may assume  the diameter $|T_l(\pi Z_l')|$ remains uniformly bounded away from 0 since otherwise $E$ is a single point. Moreover, using \eqref{yydecay},   there is a  constant $b\geq 1$ such that, for all $n$ such that $T_l(\pi(x_n))  \in B(0,1)$,
\[
T_l(\pi(x_n)) \ \in \ B(T_l(\pi(z_n)), b3^{-n}).
\]
It follows  that there is a non-trivial affine map $f \colon \mathbb{R} \to \mathbb{R}$ such that $E$ consists of a finite set together with a set which is contained in the closure of 
\[
 \bigcup_{n\geq 1} B(f(2^{-n}), b3^{-n})
\]
with at most one point in each ball.   In particular, $\hd E = \ad E = 0$.

Finally, we are left with the case when 
\begin{equation} \label{goodcase}
\frac{|\pi F_{\pi(l),k(l)}|}{|\pi(l) F_{\pi(l),k(l)}|} \ \to \ 1
\end{equation}
as $l \to \infty$, which  is similar to the proof of the lower bound, above.  In fact, if $\|T_l\| / h(g(\pi(l),k(l)))$ is uniformly bounded above, then the proof proceeds exactly as in the lower bound and we find that $E$ is a weak tangent to $Y(c_0)$ for some $c_0 \leq c$ (using upper semi-continuity) and therefore has dimension at most $\phi(\pi)$.  Therefore, by taking a subsequence we can assume
\[
\frac{\|T_l\|}{h(g(\pi(l),k(l)))} \ \to \ \infty.
\]
This corresponds to the maps $T_l$ only focusing on increasingly small parts of $\pi F_{\pi(l),k(l)}$ and not the whole set.   Let  $c_l$ be   such that 
\[
\phi(\pi_{k(l)}) \ = \ \frac{\log 2}{-\log c_l}.
\]
Upper semi-continuity implies that $\limsup_{l \to \infty} \phi(\pi_{k(l)})   \leq \phi(\pi)$ and therefore we may assume $c_l \to c_0 \leq c$.  Recall that
\[ 
Y_n (c_l) \ = \ \bigcup_{m=1}^n \left( 9^{-2^m}+16^{-2^m}E_m(c_l)\right)
\]
and consider the corresponding decomposition of $\pi F_{\pi(l),k(l)}$ into `clusters'  corresponding to points in $E_m(c_l)$.  Since the diameter of these clusters is shrinking much faster than the corresponding translations, if each $T_l$ `sees' more than once cluster (that is, the $T_l$-images of more than one cluster  intersect  $B(0,1)$), then the relative diameter of the $T_l$ images of the clusters intersecting $B(0,1)$ approaches 0 in the limit and $E$ is countable and therefore has Hausdorff dimension 0.  Therefore the only way for $E$ to have Hausdorff dimension strictly larger than 0 is if each $T_l$ only `sees' one cluster in $\pi F_{\pi(l),k(l)}$.  We may therefore assume that $T_l(\pi F_{\pi(l),k(l)}) \cap B(0,1)$ consists of a cluster of $2^{m(l)}$ points corresponding to  the projection of the $m(l)$th level approximation of the self-similar set  $E(c_l)$ which has dimension $\phi(\pi_{k(l)})$.

Using \eqref{goodcase} we claim that this sequence of clusters   approaches an affine copy of the  self-similar set  $E(c_0)$ union an affine copy of $Z$ in the Hausdorff metric.  This finishes the proof, since such a set has Hausdorff dimension at most $\phi(\pi)$. 

To justify our claim, consider the points in the cluster labelled as $\{p_1, p_2, \dots, p_{2^{m(l)}}\} \subset T_l(\pi F_{\pi(l),k(l)}) \cap B(0,1)$ according to their position in $F_{\pi(l),k(l)}$ from `top to bottom' and compare these to the corresponding points in an appropriately chosen  affine image of $E_{m(l)}(c_0)$, labelled accordingly as $\{e_1, e_2, \dots, e_{2^{m(l)}}\}$.  The right affine map to choose is $f \colon \mathbb{R} \to \mathbb{R}$ given by $f(w) = aw+b$ where $a \in (0,2]$ is the diameter of $\pi(l) (\{p_1, p_2, \dots, p_{2^{m(l)}}\})$ and $b$ is chosen such that  $p_{2^{m(l)}}= e_{2^{m(l)}}$.

\begin{figure}[t]
\centering
\begin{tikzpicture}[scale=1]
  \draw (0,2.5) -- (0,0) -- (2,0) -- (2,2.5);
  \draw [dotted] (0,2.5) -- (0,3);
  \draw [dotted] (2,2.5) -- (2,3);
  \draw (0,3) -- (0,5.5) -- (2,5.5) -- (2,3);

  \draw [rotate=-7] (-1.5,3.5) -- (-1.5,2.25) node [left] {$\pi$};
  \draw [rotate=-7,->] (-1.5,2.25) -- (-1.5,1);

  \draw [rotate=-7] (-2,-1) -- (3,-1);

  \node at (2,5.5) [circle,fill,inner sep=1pt]{};
  \node at (1.7,4.2) [circle,fill,inner sep=1pt]{};
  \node at (1.3,3.5) [circle,fill,inner sep=1pt]{};
  \node at (1,3.3) [circle,fill,inner sep=1pt]{};

  \draw [dotted,xshift=2 cm,yshift=5.5 cm,rotate=-7] (0,0) -- (0,-6.7);
  \draw [dotted,xshift=1 cm,yshift=3.3 cm,rotate=-7] (0,0) -- (0,-4.4);

  \node at (0.4,2.5) [circle,fill,inner sep=1pt]{};
  \node at (0.37,1.3) [circle,fill,inner sep=1pt]{};
  \node at (0.33,0.625) [circle,fill,inner sep=1pt]{};
  \node at (0.3,0.3) [circle,fill,inner sep=1pt]{};

  \draw [dotted,xshift=0.4 cm,yshift=2.5 cm,rotate=-7] (0,0) -- (0,-3.5);
  \draw [dotted,xshift=0.3 cm,yshift=0.3 cm,rotate=-7] (0,0) -- (0,-1.3);
  \end{tikzpicture}
  \caption{The projection of points in the upper part of the rectangle  approximate $Z$ and the projection of points in the lower part approximate $E(c_0)$. The lower part of the rectangle has been magnified for illustrative purposes.}
  \label{fig:illustration2}
\end{figure}
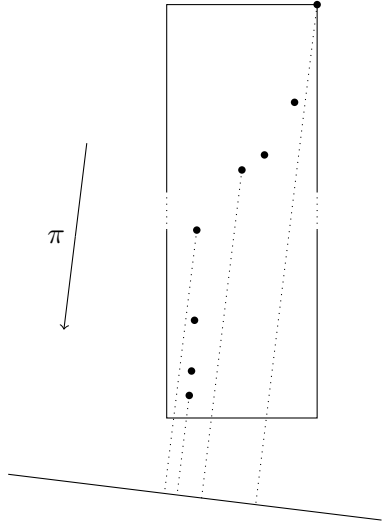

The pointwise error decreases as the label increases, but we cannot control the error between the first pair of points although it must be bounded above by 2 since both lie in $B(0,1)$.   (The error uniformly  approaches 0 when considering the corresponding points in $\pi F_{\pi(l),k(l)}$, but applying $T_l$ magnifies the error.)  However, what we can say is that the error decreases geometrically, with
\[
|p_t-e_t| \ \leq \ 2^{2-t}.
\] 
This follows from the positioning of the points in $ F_{\pi(l),k(l)}$ decaying geometrically in the direction orthogonal to the subspace $\pi(l)$.  Consider the decomposition of the cluster as
\[
\{p_1, p_2, \dots, p_{ m(l)} \} \cup \{p_{ m(l)+1}, p_2, \dots, p_{2^{m(l)}}\}
\]
and note that the set on the left lies within $m(l)2^{-m(l)}$ of an affine image of $Z$ and the set on the right lies within $2^{-m(l)+1}$ of an affine image of $E(c_l)$.  We may conclude that $E$ is a weak tangent of the union of an affine copy of $E(c_0)$ and an affine copy of $Z$. \hfill \qed

\section{Final remarks} \label{discussion}

We proved that for any upper semi-continuous function  $\phi\colon G(1,2) \to [0,1]$ there exists a compact set $F \subset \mathbb{R}^2$ such that $\dim_\textup{A} \pi F = \phi(\pi)$ for all $\pi \in G(1,2) $.  Notice, however, that the function $\pi \mapsto \ad \pi F$ need not be upper semi-continuous in general since a line segment satisfies $\dim_\textup{A} \pi F = 1$ for all but one subspace $\pi$ where the dimension is 0.  It is therefore natural to ask what sort of functions can be realised as $\pi \mapsto \ad \pi F$.  We observe that the set
\[
\{ \phi\colon G(1,2) \to [0,1] \ : \  \phi(\pi) = \dim_\textup{A} \pi F  \text{ for some compact $F \subset \mathbb{R}^2$} \}
\]
has cardinality $2^{\aleph_0}$: Theorem \ref{main1} implies that it has cardinality at least $2^{\aleph_0}$ but it cannot be more than $2^{\aleph_0}$ since there is a natural surjection from the set of compact subsets of $ \mathbb{R}^2$ onto this set.  In particular, there must be functions  which \emph{cannot} be realised as $\pi \mapsto \ad \pi F$ for a compact set $F$, since the set of all functions has cardinality strictly greater than $2^{\aleph_0}$.

\vspace{6mm}


\begin{centering}

\textbf{Acknowledgements}

JMF was financially supported by  a \emph{Leverhulme Trust Research Fellowship} (RF-2016-500) and  an \emph{EPSRC Standard Grant} (EP/R015104/1). AK was financially supported by the \emph{Finnish Center of Excellence in Analysis and Dynamics Research} and the aforementioned EPSRC grant, which covered the local costs of the visit of AK to the University of St Andrews in May 2018 where this research began. The authors thank Eino Rossi for pointing out a simple improvement which got rid of a minor additional assumption used in an earlier version of the paper. The  authors also thank Tuomas Orponen for many helpful discussions on the topic. Finally, the authors thank an anonymous referee for making several useful suggestions.
\end{centering}


\begin{multicols}{2}{

\noindent \emph{Jonathan M. Fraser\\
School of Mathematics and Statistics\\
The University of St Andrews\\
St Andrews, KY16 9SS, Scotland} \\

\noindent  Email: jmf32@st-andrews.ac.uk\\ \\

\noindent \emph{Antti K\"aenm\"aki\\
Department of Physics and Mathematics\\
University of Eastern Finland\\
P.O. Box 111, FI-80101 Joensuu, Finland} \\

\noindent  Email: antti.kaenmaki@uef.fi\\ \\
}

\end{multicols}

\end{document}